\documentclass{amsart}

\usepackage[utf8]{inputenc}

\usepackage{amssymb,amsmath}

\usepackage[T1]{fontenc}

\allowdisplaybreaks[4]

\theoremstyle{plain}% Theorem-like structures provided by amsthm.sty
\newtheorem{theorem}{Theorem}[section]
\newtheorem{lemma}[theorem]{Lemma}
\newtheorem{corollary}[theorem]{Corollary}

\theoremstyle{definition}

\theoremstyle{remark}
\newtheorem{remark}{Remark}

\begin{document}

\title[Orthogonally additive polynomials]{Orthogonally additive polynomials on the algebras of approximable operators}

\author{J. Alaminos}
\address{Departamento de An\' alisis
Matem\' atico\\ Fa\-cul\-tad de Ciencias\\ Universidad de Granada\\
18071 Granada, Spain} 
\email{alaminos@ugr.es}
\author{M. L. C. Godoy}
\address{Departamento de An\' alisis
Matem\' atico\\ Fa\-cul\-tad de Ciencias\\ Universidad de Granada\\
18071 Granada, Spain} 
\email{mlcg.godoy@gmail.com}
\author{A.\,R. Villena}
\address{Departamento de An\' alisis
Matem\' atico\\ Fa\-cul\-tad de
 Ciencias\\ Universidad de Granada\\
18071 Granada, Spain} 
\email{avillena@ugr.es}

\begin{abstract}
Let $X$ and $Y$ be Banach spaces, let $\mathcal{A}(X)$ stands for the algebra of approximable operators on $X$, and let $P\colon\mathcal{A}(X)\to Y$ be an orthogonally additive, continuous $n$-homogeneous polynomial. If $X^*$ has the bounded approximation property, then we show that there exists a unique continuous linear map $\Phi\colon\mathcal{A}(X)\to Y$ such that $P(T)=\Phi(T^n)$ for each $T\in\mathcal{A}(X)$.
\end{abstract}



\thanks{The first and the third named authors were supported by MINECO grant MTM2015--65020--P and Junta de Andaluc\'{\i}a grant FQM--185. The second named author was supported by Beca de iniciaci\'on a la investigaci\'on of Universidad de Granada.}

\maketitle

\section{Introduction}

Throughout all  algebras and linear spaces are complex.
Of course, linearity is understood to mean complex linearity.

Let $A$ be an algebra and let $Y$ be a linear space.
A map $P\colon A\to Y$ is said to be \emph{orthogonally additive}
if
\[
a,b\in A, \ ab=ba=0  \ \Rightarrow \ P(a+b)=P(a)+P(b) .
\]
Let $X$ and $Y$ be linear spaces, and let $n\in\mathbb{N}$. A map
$P\colon X\to Y$ is said to be an \emph{$n$-homogeneous polynomial} if
there exists an $n$-linear map $\varphi\colon X^n\to Y$ such that
$P(x)=\varphi(x,\dotsc,x)$ $(x\in X)$. Here and subsequently, $X^n$
stands for the $n$-fold Cartesian product of $X$. Such a map is unique
if it is required to be symmetric. This is a consequence of the
so-called polarization formula which defines $\varphi$ through
\[
\varphi(x_1,\ldots,x_n)=
\frac{1}{n!\,2^n}\sum_{\epsilon_1,\ldots,\epsilon_n=\pm 1}
\epsilon_{1}\cdots\epsilon_{n} P(\epsilon_1x_1+\cdots+\epsilon_{n}x_{n}).
\]
Further, in the case where $X$ and $Y$ are normed spaces, the
polynomial $P$ is continuous if and only if the symmetric $n$-linear
map $\varphi$ associated with $P$ is continuous.

Let $A$ be a Banach algebra.
Given $n\in\mathbb{N}$, a Banach space $Y$, and a continuous linear map
$\Phi\colon A\to Y$, the map $a\mapsto \Phi(a^n)$ is a typical example
of continuous orthogonally additive $n$-homogeneous polynomial,
and a standard problem consists in determining whether these are
precisely the canonical examples of continuous orthogonally
additive $n$-homogeneous polynomials on $A$.
In the case where $A$ is a $C^*$-algebra, it is shown in~\cite{P}
that every continuous $n$-homogeneous polynomial
$P\colon A\to Y$ can be represented in the form
\begin{equation*}
P(a)=\Phi(a^n) \quad  (a\in A)
\end{equation*}
for some continuous linear map $\Phi\colon A\to Y$
(see~\cite{P2,P3} for the case where $A$ is a $C^*$-algebra and $P$ is
a holomorphic map). The references~\cite{A1,A2,V} discuss the case
where~$A$ is a commutative Banach algebra. This paper is concerned with the
problem of representing the  continuous orthogonally additive homogeneous
polynomials in the case where $A$ is the algebra $\mathcal{A}(X)$ of
\emph{approximable operators} on a Banach space $X$. Here,
$\mathcal{B}(X)$ is the Banach algebra of continuous linear operators on $X$,
$\mathcal{F}(X)$ is the two-sided ideal of $\mathcal{B}(X)$ consisting of
finite-rank operators, and $\mathcal{A}(X)$ is the closure of
$\mathcal{F}(X)$ in $\mathcal{B}(X)$ with respect to the operator norm.

Let $X$ be a Banach space. Then $X^*$ denotes the dual of $X$. For $x\in X$
and $f\in X^*$, we write $x\otimes f$  for the operator defined by
$(x\otimes f)(y)=f(y)x$ for each $y\in X$. Let $n\in\mathbb{N}$. Then we
write $\mathbb{M}_n$ for the full matrix algebra of order $n$ over $\mathbb{C}$,
and $\mathfrak{S}_n$ for the symmetric group of order $n$.

\section{Orthogonally additive polynomials on the algebra of finite-rank operators}

\begin{lemma}\label{l1526}
Let $\mathcal{M}$ be a Banach algebra isomorphic to $\mathbb{M}_k$ for
some $k\in\mathbb{N}$, let $Y$ be a Banach space, and
let $P\colon\mathcal{M}\to Y$ be an orthogonally additive $n$-homogeneous
polynomial. Then there exists a unique linear map $\Phi\colon\mathcal{M}\to
Y$ such that
\begin{equation}\label{m00}
P(a)=\Phi(a^n)
\end{equation}
for each $a\in\mathcal{M}$.
Further, if $\varphi\colon\mathcal{M}^n\to Y$ is the symmetric $n$-linear
map associated with $P$ and $e$ is the identity of $\mathcal{M}$, then
\begin{equation}\label{m0}
\Phi(a)=\varphi(a,e,\dotsc,e)
\end{equation}
for each $a\in \mathcal{M}$.
\end{lemma}

\begin{proof}
Let $\Psi\colon\mathcal{M}\to\mathbb{M}_k$ be an isomorphism.
Since $\mathbb{M}_k$ is a $C^*$-algebra and
the map $P\circ\Psi^{-1}\colon\mathbb{M}_k\to Y$ is easily seen to be
an orthogonally additive $n$-homogeneous polynomial,~\cite[Corollary~3.1]{P}
then shows that there exists
a unique linear map $\Theta\colon\mathbb{M}_k\to Y$ such that
$P(\Psi^{-1}(M))=\Theta(M^n)$ for each $M\in\mathbb{M}_k$. It is a simple matter to check
that the map $\Phi=\Theta\circ\Psi$ satisfies~\eqref{m00}.
Now the polarization of~\eqref{m00} yields
\begin{equation*}
\varphi(a_1,\ldots,a_n)=
\frac{1}{n!}
\Phi\left(\sum_{\sigma\in\mathfrak{S}_n}a_{\sigma(1)}\cdots a_{\sigma(n)}\right)
\end{equation*}
for each $(a_1,\ldots,a_n)\in \mathcal{M}$,
whence
$\varphi(a,e,\dotsc,e)=\Phi(a)$ for each $a\in\mathcal{M}$.
\end{proof}

\begin{lemma}\label{l1525}
Let $X$ be a Banach space and let $T_1,\dotsc,T_m\in\mathcal{F}(X)$.
Then there exists a subalgebra $\mathcal{M}$ of $\mathcal{F}(X)$ such that
$T_1,\dotsc,T_m\in\mathcal{M}$ and $\mathcal{M}$ is isomorphic
to $\mathbb{M}_k$ for some $k\in\mathbb{N}$.
\end{lemma}

\begin{proof}
We can certainly assume that $T_1,\dotsc,T_m$ are rank-one operators.
Write $T_j=x_j\otimes f_j$ with $x_j\in X$ and $f_j\in X^*$ for each $j\in \left\{1,\dotsc,m\right\}$.

We claim that there exist $y_1,\dotsc,y_k\in X$ and $g_1,\dotsc g_k\in X^*$
such that
\begin{equation}\label{e1559}
\begin{split}
x_1,\dotsc,x_m\in\text{span}\left(\{y_1,\dotsc,y_k\}\right)&,\\
f_1,\dotsc,f_m\in\text{span}\left(\{g_1,\dotsc,g_k\}\right)&,
\end{split}
\end{equation}
and
\begin{equation}\label{e1558}
g_i(y_j)=\delta_{ij} \quad \left(i,j\in\{1,\dotsc,k\}\right).
\end{equation}
Let $\{g_1,\dotsc,g_l\}$ be a basis of the linear span of $\{f_1,\dotsc,f_m\}$,
and let $y_1\dotsc,y_l\in X$ be such that $g_i(y_j)=\delta_{ij}$
($i,j\in\{1,\dotsc,l\}$).
Let $U$ be the linear span of $\{y_1,\dotsc,y_l\}$ in $X$,
let $Q\colon X\to X/U$ the quotient map,
and let
\[
U^\perp=\{f\in X^*\colon f(y_1)=\dots=f(y_l)=0\}.
\]
If $\{x_1,\dotsc,x_m\}\subset\{y_1,\dotsc,y_l\}$, then our claim follows.
We now assume that $\{x_1,\dotsc,x_m\}\not\subset\{y_1,\dotsc,y_l\}$.
Let $y_{l+1},\dotsc,y_k\in X$ be such that
$\{Q(y_{l+1}),\dotsc,Q(y_k)\}$ is a basis of the linear span of the
set $\{Q(x_1),\dotsc,Q(x_m)\}$ in $X/U$. Since the map $f\mapsto f\circ Q$
defines an isometric isomorphism from $\bigl(X/U\bigr)^*$ onto $U^\perp$,
it follows that there exist $g_{l+1},\dotsc,g_k\in U^\perp$ such that
$g_i(y_j)=\delta_{ij}$ for all $i,j\in\{l+1,\dotsc,k\}$.
It is a simple matter to check that the sets $\{y_1,\dotsc,y_k\}$ and
$\{g_1,\dotsc,g_k\}$ satisfy the requirements \eqref{e1559} and \eqref{e1558}.

Let $\mathcal{M}$ be the subalgebra of $\mathcal{F}(X)$ generated by the set
$\bigl\{y_i\otimes g_j\colon i,j\in\{1,\dotsc,k\}\bigr\}$.
By \eqref{e1559}, $T_1,\dotsc,T_m\in\mathcal{M}$.
From~\eqref{e1558} we conclude that the algebra $\mathcal{M}$ is isomorphic
to the full matrix algebra $\mathbb{M}_n$. Actually, the map $T\mapsto\bigl[g_i(T(y_j))\bigr]_{i,j}$
defines an isomorphism from $\mathcal{M}$ onto $\mathbb{M}_k$ that takes
the operator $y_i\otimes g_j$ into the standard matrix unit $E_{i,j}$ ($i,j\in\{1,\dotsc,k\}$).
\end{proof}

From Lemmas~\ref{l1526} and~\ref{l1525} we see immediately that each
orthogonally additive $n$-homogeneous polynomial $P$ on $\mathcal{F}(X)$
can be represented in the standard way on any finitely generated 
subalgebra of $\mathcal{F}(X)$. The issue is whether the pieces 
$\Phi_\mathcal{M}$ (where $\mathcal{M}$ ranges over the finitely 
generated subalgebras of $\mathcal{F}(X)$) fit together to give a linear map $\Phi$ representing the polynomial $P$ on the whole $\mathcal{F}(X)$.

\begin{corollary}\label{crf}
Let $X$ and $Y$ be Banach spaces, and let $P\colon \mathcal{F}(X)\to Y$ be an orthogonally additive $n$-homogeneous polynomial.
Then there exists a unique linear map $\Phi\colon\mathcal{F}(X)\to Y$ such that
$P(T)=\Phi(T^n)$ for each $T\in\mathcal{F}(X)$.
\end{corollary}

\begin{proof}
Let $T\in\mathcal{F}(X)$. Lemma~\ref{l1525} shows that
there exists a subalgebra $\mathcal{M}$ of $\mathcal{F}(X)$ such that
$T\in\mathcal{M}$ and $\mathcal{M}$ is isomorphic to $\mathbb{M}_k$ for some $k\in\mathbb{N}$.
Then Lemma~\ref{l1526} yields a unique linear map $\Phi_\mathcal{M}\colon\mathcal{M}\to Y$ such that
$P(S)=\Phi_{\mathcal{M}}(S^n)$ for each $S\in\mathcal{M}$.
Then we set $\Phi(T)=\Phi_\mathcal{M}(T)$. We now show that $\Phi$ is well-defined.
Assume that $\mathcal{M}_1$ and $\mathcal{M}_2$ are subalgebras of $\mathcal{F}(X)$ with the properties that
$T\in\mathcal{M}_j$ and $\mathcal{M}_j$ is isomorphic to a full matrix algebra ($j=1,2$).
Then, according to Lemma~\ref{l1525}, there exists
a subalgebra $\mathcal{N}$ of $\mathcal{F}(X)$ such that
$\mathcal{M}_1,\mathcal{M}_2\subset\mathcal{N}$ and $\mathcal{N}$ is isomorphic to a full matrix algebra.
The uniqueness of the represention asserted in Lemma~\ref{l1526} gives that $\Phi_\mathcal{N}$ equals $\Phi_{\mathcal{M}_j}$
when restricted to $\mathcal{M}_j$ for $j=1,2$. Accordingly, we have
$\Phi_{\mathcal{M}_1}(T)=\Phi_\mathcal{N}(T)=\Phi_{\mathcal{M}_2}(T)$.

Let $S$, $T\in\mathcal{F}(X)$ and let $\alpha,\beta\in\mathbb{C}$.
By Lemma~\ref{l1525}
there exists a subalgebra $\mathcal{M}$ of $\mathcal{F}(X)$ such that
$S,T\in\mathcal{M}$ and $\mathcal{M}$ is isomorphic to a full matrix algebra.
Then
\[
\Phi(\alpha S+\beta T)=
\Phi_\mathcal{M}(\alpha S+\beta T)=
\alpha\Phi_\mathcal{M}(S)+\beta\Phi_\mathcal{M}(T)=
\alpha\Phi(S)+\beta\Phi(T),
\]
and, since $T^n\in\mathcal{M}$,
\[
P(T)=\Phi_\mathcal{M}(T^n)=\Phi(T^n).
\]
This shows that $\Phi$ is linear and gives a representation of $P$.
It should be pointed out that the polarization of this representations
yields
\begin{equation}\label{m1}
\varphi(T_1,\ldots,T_n)=
\frac{1}{n!}
\Phi\left(\sum_{\sigma\in\mathfrak{S}_n}T_{\sigma(1)}\cdots T_{\sigma(n)}\right)
\end{equation}
for each $(T_1,\ldots,T_n)\in \mathcal{F}(X)$,
where $\varphi$ is the symmetric $n$-linear map associated with $P$.

Our final task is to prove the uniqueness of the map $\Phi$.
Suppose that $\Psi\colon\mathcal{F}(X)\to Y$ is a linear map
such that $P(T)=\Psi(T^n)$ for each $T\in\mathcal{F}(X)$.
The polarization of this identity gives
\begin{equation}\label{m2}
\varphi(T_1,\ldots,T_n)=
\frac{1}{n!}
\Psi\left(\sum_{\sigma\in\mathfrak{S}_n}T_{\sigma(1)}\cdots T_{\sigma(n)}\right)
\end{equation}
for each $(T_1,\ldots,T_n)\in \mathcal{F}(X)$.
Let $T\in\mathcal{F}(X)$. On account of Lemma~\ref{l1525}, there exists
$S\in\mathcal{F}(X)$ such that $TS=ST=T$. From \eqref{m1} and \eqref{m2}
we obtain $\Phi(T)=\varphi(T,S,\dotsc,S)=\Psi(T)$.
\end{proof}

It is not clear at all whether or not the linear map $\Phi$ given in the preceding result is continuous
in the case where the polynomial $P$ is continuous.

\section{Orthogonally additive polynomials on the algebra of approximable operators}

Let $A$ be a Banach algebra,
let $(e_\lambda)_{\lambda\in\Lambda}$ be a bounded approximate identity for $A$ of bound $C$, and
let $\mathcal{U}$ be an ultrafilter on $\Lambda$ containing the order filter on $\Lambda$
(which will be associated with $(e_\lambda)_{\lambda\in\Lambda}$ and fixed throughout).
Let $Y$ be a dual Banach space and let $Y_*$ be a predual of $Y$.
It follows from the Banach-Alaoglu theorem that each bounded subset of $Y$
is relatively compact with respect to the $\sigma(Y,Y_*)$-topology on $Y$.
Consequently, each bounded net $(y_\lambda)_{\lambda\in\Lambda}$ in $Y$
has a unique limit with respect to the $\sigma(Y,Y_*)$-topology
along the ultrafilter $\mathcal{U}$, and we write $\lim_{\mathcal{U}} y_\lambda$ for this limit.

Let $\varphi\colon A^n\to Y$ be a continuous $n$-linear map.
For each $a_1,\ldots,a_{n-1}\in A$ and $\lambda\in\Lambda$, we have
\begin{equation}\label{e1908}
\begin{split}
\Vert\varphi(a_1,\ldots,a_{n-1},e_\lambda)\Vert
& \le
\Vert\varphi\Vert\Vert a_1\Vert\cdots\Vert a_{n-1}\Vert\Vert e_\lambda\Vert\\
&\le
C\Vert\varphi\Vert\Vert a_1\Vert\cdots\Vert a_{n-1}\Vert .
\end{split}
\end{equation}
Hence the net $(\varphi(a_1,\ldots,a_{n-1},e_\lambda))_{\lambda\in\Lambda}$ is bounded
and therefore we can define the map $\varphi'\colon A^{n-1}\to Y$ by
\[
\varphi'(a_1,\ldots,a_{n-1})=\lim_{\mathcal{U}}\varphi(a_1,\ldots,a_{n-1},e_\lambda)
\]
for each $(a_1,\ldots,a_{n-1})\in A^{n-1}$.
The linearity of the limit along an ultrafilter on a topological linear space gives
the $(n-1)$-linearity of $\varphi'$. Moreover, from~\eqref{e1908} we deduce that
\[
\Vert\varphi'(a_1,\ldots,a_{n-1})\Vert\le C\Vert\varphi\Vert\Vert a_1\Vert\cdots\Vert a_{n-1}\Vert
\]
for each $(a_1,\ldots,a_{n-1})\in A^{n-1}$, which gives the continuity of $\varphi'$ and
$\Vert\varphi'\Vert\le C\Vert\varphi\Vert$. Further,
it is clear that if the map $\varphi$ is symmetric, then the map $\varphi'$ is symmetric.

\begin{lemma}\label{l177}
Let $A$ be a Banach algebra with a bounded approximate identity $(e_\lambda)_{\lambda\in\Lambda}$,
let $Y$ be a dual Banach space, and let $\varphi\colon A^n\to Y$ be
a continuous symmetric $n$-linear map with $n\ge 2$.
Suppose that
\[
\varphi(a_1,\ldots,a_n)=
\frac{1}{n!}
\sum_{\sigma\in\mathfrak{S}_n}\varphi'(a_{\sigma(1)},\ldots,a_{\sigma(n-1)}a_{\sigma(n)})
\]
for each $(a_1,\dotsc,a_n)\in A^n$.
Then there exists a continuous linear map $\Phi\colon A\to Y$ such that
\[
\varphi(a_1,\ldots,a_n)=
\frac{1}{n!}
\Phi\left(\sum_{\sigma\in\mathfrak{S}_n}a_{\sigma(1)}\cdots a_{\sigma(n)}\right)
\]
for each $(a_1,\ldots,a_n)\in A^n$.
\end{lemma}

\begin{proof}
The proof is by induction on $n$.
The result is certainly true if $n=2$.
Assume that the result is true for $n$, and let
$\varphi\colon A^{n+1}\to Y$ be a continuous symmetric $(n+1)$-linear map such that
\[
\varphi(a_1,\ldots,a_n,a_{n+1})=
\frac{1}{(n+1)!}
\sum_{\sigma\in\mathfrak{S}_{n+1}}\varphi'(a_{\sigma(1)},\ldots,a_{\sigma(n-1)},a_{\sigma(n)}a_{\sigma(n+1)})
\]
for each $(a_1,\dotsc,a_{n+1})\in A^{n+1}$.
We claim that
\begin{equation}\label{1112}
\varphi'(a_1,\ldots,a_n)=
\frac{1}{n!}
\sum_{\sigma\in\mathfrak{S}_n}\varphi''(a_{\sigma(1)},\ldots,a_{\sigma(n-1)}a_{\sigma(n)})
\end{equation}
for each $(a_1,\dotsc,a_{n})\in A^n$.
Here $\varphi''$ stands for the $(n-1)$-linear map $(\varphi')'$. Indeed,
for all $(a_1,\dotsc,a_n)\in A^n$ and $\lambda\in\Lambda$, we have
\begin{align*}
\varphi(a_1,\ldots,a_n,e_\lambda)
& =
\frac{1}{(n+1)!}\sum_{\tau\in\mathfrak{S}_{n}}\varphi'(a_{\tau(1)},\ldots,a_{\tau(n-1)},a_{\tau(n)}e_\lambda)\\
& \quad {}+ \frac{1}{(n+1)!}\sum_{\tau\in\mathfrak{S}_{n}}\varphi'(a_{\tau(1)},\ldots,a_{\tau(n-1)},e_\lambda a_{\tau(n)})\\
& \quad {}+
\frac{1}{(n+1)!}\sum_{\tau\in\mathfrak{S}_{n}}\varphi'(a_{\tau(1)},\ldots,e_\lambda ,a_{\tau(n-1)}a_{\tau(n)})\\
&\quad {}+\dotsb+
\frac{1}{(n+1)!}\sum_{\tau\in\mathfrak{S}_{n}}\varphi'(e_\lambda,a_{\tau(1)},\ldots ,a_{\tau(n-1)}a_{\tau(n)}).\\
\end{align*}
Since $(e_\lambda)_{\lambda\in\Lambda}$ is a bounded approximate identity for $A$,
the nets $(a_ke_\lambda)_{\lambda\in\Lambda}$ and $(e_\lambda a_k)_{\lambda\in\Lambda}$
converge to $a_k$ in norm for each $k\in\{1,\dotsc,n\}$, and so, taking limits along $\mathcal{U}$
on both sides of the above equation
(and using the continuity of $\varphi'$), we see that
\begin{align*}
\varphi'(a_1,\ldots,a_n)
& =
\lim_\mathcal{U}
\varphi(a_1,\ldots,a_n,e_\lambda)\\
& =
\frac{1}{(n+1)!}\sum_{\tau\in\mathfrak{S}_{n}}\lim_\mathcal{U}\varphi'(a_{\tau(1)},\ldots,a_{\tau(n-1)},a_{\tau(n)}e_\lambda)\\
& \quad {}+
\frac{1}{(n+1)!}\sum_{\tau\in\mathfrak{S}_{n}}\lim_\mathcal{U}\varphi'(a_{\tau(1)},\ldots,a_{\tau(n-1)},e_\lambda a_{\tau(n)})\\
& \quad {}+
\frac{1}{(n+1)!}\sum_{\tau\in\mathfrak{S}_{n}}\lim_\mathcal{U}\varphi'(a_{\tau(1)},\ldots,e_\lambda ,a_{\tau(n-1)}a_{\tau(n)})\\
& \quad {}+\dotsb +
\frac{1}{(n+1)!}\sum_{\tau\in\mathfrak{S}_{n}}\lim_\mathcal{U}\varphi'(e_\lambda,a_{\tau(1)},\ldots ,a_{\tau(n-1)}a_{\tau(n)})\\
& =
\frac{1}{(n+1)!}\sum_{\tau\in\mathfrak{S}_{n}}\varphi'(a_{\tau(1)},\ldots,a_{\tau(n-1)},a_{\tau(n)})\\
& \quad {}+
\frac{1}{(n+1)!}\sum_{\tau\in\mathfrak{S}_{n}}\varphi'(a_{\tau(1)},\ldots,a_{\tau(n-1)},a_{\tau(n)})\\
& \quad {}+
\frac{1}{(n+1)!}\sum_{\tau\in\mathfrak{S}_{n}}\varphi''(a_{\tau(1)},\ldots ,a_{\tau(n-1)}a_{\tau(n)})\\
& \quad {}+\dotsb+ 
\frac{1}{(n+1)!}\sum_{\tau\in\mathfrak{S}_{n}}\varphi''(a_{\tau(1)},\ldots ,a_{\tau(n-1)}a_{\tau(n)})\\
& =
\frac{1}{(n+1)!} 2n!\,\varphi'(a_1,\ldots,a_{n-1},a_n)\\
& \quad {}+
\frac{1}{(n+1)!}(n-1)\sum_{\tau\in\mathfrak{S}_{n}}\varphi''(a_{\tau(1)},\ldots ,a_{\tau(n-1)}a_{\tau(n)}).\\
\end{align*}
We thus get
\[
\Bigl(1-\frac{2}{n+1}\Bigr)\varphi'(a_1,\ldots,a_n)=
\frac{n-1}{(n+1)!}\sum_{\tau\in\mathfrak{S}_{n}}\varphi''(a_{\tau(1)},\ldots ,a_{\tau(n-1)}a_{\tau(n)}),
\]
which proves our claim.

By~\eqref{1112} and the inductive hypothesis,
there exists a continuous linear map $\Phi\colon A\to Y$ such that
\[
\varphi'(a_1,\ldots,a_n)=
\frac{1}{n!}\Phi\left(
\sum_{\tau\in\mathfrak{S}_n}a_{\tau(1)}\cdots a_{\tau(n)}\right)
\]
and thus
\begin{align*}
\varphi'(a_1,\ldots,a_n) & = 
\frac{1}{n!} \sum_{\tau\in\mathfrak{S}_{n-1}}\Phi(a_{\tau(1)}\cdots a_{\tau(n-1)}a_n)+
\frac{1}{n!}\sum_{\tau\in\mathfrak{S}_{n-1}}\Phi(a_{\tau(1)}\cdots a_na_{\tau(n-1)})\\
& \quad  {}+\cdots +
\frac{1}{n!}\sum_{\tau\in\mathfrak{S}_{n-1}}\Phi\left(a_na_{\tau(1)}\cdots a_{\tau(n-1)}\right)
\end{align*}
for each $(a_1,\ldots,a_n)\in A^n$.
Therefore, for each $(a_1,\ldots,a_n,a_{n+1})\in A^{n+1}$, we have
\begin{multline*}
\varphi(a_1,\ldots,a_n,a_{n+1})
=
\frac{1}{(n+1)!}\sum_{\sigma\in\mathfrak{S}_{n+1}}
\varphi'(a_{\sigma(1)},\ldots,a_{\sigma(n-1)},a_{\sigma(n)}a_{\sigma(n+1)})\\
=
\frac{1}{(n+1)!}\sum_{\sigma\in\mathfrak{S}_{n+1}}
\frac{1}{n!}\sum_{\tau\in\mathfrak{S}_{n-1}}
\Phi(a_{\sigma(\tau(1))}\cdots a_{\sigma(\tau(n-1))}a_{\sigma(n)}a_{\sigma(n+1)})\\
+
\frac{1}{(n+1)!}\sum_{\sigma\in\mathfrak{S}_{n+1}}
\frac{1}{n!}\sum_{\tau\in\mathfrak{S}_{n-1}}
\Phi(a_{\sigma(\tau(1))}\cdots a_{\sigma(n)}a_{\sigma(n+1)}a_{\sigma(\tau(n-1))})\\
\qquad {}+\cdots+
\frac{1}{(n+1)!}\sum_{\sigma\in\mathfrak{S}_{n+1}}
\frac{1}{n!}\sum_{\tau\in\mathfrak{S}_{n-1}}
\Phi(a_{\sigma(n)}a_{\sigma(n+1)}a_{\sigma(\tau(1))}\ldots a_{\sigma(\tau(n-1))}) \\
=
\frac{1}{(n+1)!}\frac{1}{n!}
\sum_{\tau\in\mathfrak{S}_{n-1}}\sum_{\sigma\in\mathfrak{S}_{n+1}}
\Phi(a_{\sigma(\tau(1))}\cdots a_{\sigma(\tau(n-1))}a_{\sigma(n)}a_{\sigma(n+1)})\\
{} +
\frac{1}{(n+1)!}\frac{1}{n!}
\sum_{\tau\in\mathfrak{S}_{n-1}}\sum_{\sigma\in\mathfrak{S}_{n+1}}
\Phi(a_{\sigma(\tau(1))}\cdots a_{\sigma(n)}a_{\sigma(n+1)}a_{\sigma(\tau(n-1))})\\
\qquad {}+\cdots+
\frac{1}{(n+1)!}\frac{1}{n!}
\sum_{\tau\in\mathfrak{S}_{n-1}}\sum_{\sigma\in\mathfrak{S}_{n+1}}
\Phi(a_{\sigma(n)}a_{\sigma(n+1)}a_{\sigma(\tau(1))}\cdots a_{\sigma(\tau(n-1))})\\
=
\frac{1}{(n+1)!}\frac{1}{n!}
\sum_{\tau\in\mathfrak{S}_{n-1}}\sum_{\sigma\in\mathfrak{S}_{n+1}}
\Phi(a_{\sigma(1)}\cdots a_{\sigma(n-1)}a_{\sigma(n)}a_{\sigma(n+1)}) \\
{}+
\frac{1}{(n+1)!}\frac{1}{n!}
\sum_{\tau\in\mathfrak{S}_{n-1}}\sum_{\sigma\in\mathfrak{S}_{n+1}}
\Phi(a_{\sigma(1)}\cdots a_{\sigma(n-1)}a_{\sigma(n)}a_{\sigma(n+1)}) \\
\qquad {}+\cdots+
\frac{1}{(n+1)!}\frac{1}{n!}
\sum_{\tau\in\mathfrak{S}_{n-1}}\sum_{\sigma\in\mathfrak{S}_{n+1}}
\Phi(a_{\sigma(1)}\cdots a_{\sigma(n-1)}a_{\sigma(n)}a_{\sigma(n+1)})\\
=
\frac{1}{(n+1)!}\sum_{\sigma\in\mathfrak{S}_{n+1}}
\Phi(a_{\sigma(1)}\cdots a_{\sigma(n-1)}a_{\sigma(n)}a_{\sigma(n+1)}),
\end{multline*}
and the induction continues.
\end{proof}

\begin{lemma}\label{rfc3}
Let $\xi_1,\dotsc,\xi_n,\zeta$ be noncommuting indeterminates and let $\pi_n$ be the polynomial
defined by
\begin{equation*}%\label{l1}
\pi_n(\xi_1,\dotsc,\xi_n)=
\sum_{\sigma\in \mathfrak{S}_n}\xi_{\sigma(1)}\dotsm\xi_{\sigma(n)}.
\end{equation*}
Then the following identities hold:
\begin{multline}\label{l2}
\sum_{\sigma\in \mathfrak{S}_n}\pi_n(\xi_{\sigma(1)},\dotsc,\xi_{\sigma(n)}\zeta)
= (n-1)!\sum_{\sigma\in \mathfrak{S}_n}
\Bigl[
\xi_{\sigma(1)}\dotsm \xi_{\sigma(n-1)}\xi_{\sigma(n)}\zeta
\\ {}+
\xi_{\sigma(1)}\dotsm \xi_{\sigma(n-1)}\zeta\xi_{\sigma(n)}
+ \dotsb +
\xi_{\sigma(1)}\zeta \dotsm \xi_{\sigma(n-1)}\xi_{\sigma(n)}
\Bigr],
\end{multline}
\begin{multline}\label{l23}
\sum_{\sigma\in \mathfrak{S}_n}\pi_n(\xi_{\sigma(1)},\dotsc,\zeta\xi_{\sigma(n)})
= (n-1)!\sum_{\sigma\in \mathfrak{S}_n}
\Bigl[
\xi_{\sigma(1)}\dotsm \xi_{\sigma(n-1)}\zeta \xi_{\sigma(n)}
\\ {} +
\xi_{\sigma(1)}\dotsm \zeta\xi_{\sigma(n-1)} \xi_{\sigma(n)}
+
\cdots {}+ \zeta \xi_{\sigma(1)} \dotsm \xi_{\sigma(n-1)}\xi_{\sigma(n)}
\Bigr],
\end{multline}
and
\begin{multline}\label{l4}
\sum_{\sigma\in \mathfrak{S}_n}\pi_n(\xi_{\sigma(1)},\dotsc,\xi_{\sigma(n-1)}\xi_{\sigma(n)},\zeta) =
(n-1)!\sum_{\sigma\in \mathfrak{S}_n}
\Bigl[
\xi_{\sigma(1)}\dotsm \xi_{\sigma(n-1)}\xi_{\sigma(n)}\zeta
\\ +
\zeta \xi_{\sigma(1)} \dotsm \xi_{\sigma(n-1)}\xi_{\sigma(n)}
\Bigr]\\
+(n-2)(n-2)!\sum_{\sigma\in \mathfrak{S}_n}
\Bigl[
\xi_{\sigma(1)}\dotsm \xi_{\sigma(n-1)}\zeta \xi_{\sigma(n)}
+\dotsb+
\xi_{\sigma(1)}\zeta \dotsm \xi_{\sigma(n-1)}\xi_{\sigma(n)}
\Bigr].
\end{multline}
\end{lemma}

\begin{proof}
It is clear that
\begin{equation}\label{r14}
\pi_n(\xi_1,\dotsc,\xi_n)
=
\sum_{\tau\in \mathfrak{S}_{n-1}}
\Bigl[
\xi_{\tau(1)}\dotsm \xi_{\tau(n-1)}\xi_n
+\dotsb+
\xi_n\xi_{\tau(1)}\dotsm \xi_{\tau(n-1)}
\Bigr].
\end{equation}
Therefore, for each $\sigma\in \mathfrak{S}_n$, we have
\begin{align*}
\pi_n(\xi_{\sigma(1)},\dotsc,\xi_{\sigma(n)}\zeta)
& =
\sum_{\tau\in \mathfrak{S}_{n-1}}
\Bigl[
\xi_{\sigma(\tau(1))}\dotsm \xi_{\sigma(\tau(n-1))}\xi_{\sigma(n)}\zeta
\\ 
& \qquad {}+\cdots+
\xi_{\sigma(n)}\zeta\xi_{\sigma(\tau(1))}\dotsm \xi_{\sigma(\tau(n-1))}
\Bigr].
\end{align*}
We thus get
\begin{multline*}
\sum_{\sigma\in \mathfrak{S}_n}
\pi_n(\xi_{\sigma(1)},\dotsc,\xi_{\sigma(n)}\zeta) \\*
=
\sum_{\tau\in \mathfrak{S}_{n-1}}\sum_{\sigma\in \mathfrak{S}_n}
\Bigl[
\xi_{\sigma(\tau(1))}\dotsm \xi_{\sigma(\tau(n-1))}\xi_{\sigma(n)}\zeta
+\cdots+
\xi_{\sigma(n)}\zeta\xi_{\sigma(\tau(1))}\dotsm \xi_{\sigma(\tau(n-1))}
\Bigr] \\
=
\sum_{\tau\in \mathfrak{S}_{n-1}}
\Bigl[
\sum_{\sigma\in \mathfrak{S}_n}\xi_{\sigma(1)}\dotsm \xi_{\sigma(n-1)}\xi_{\sigma(n)}\zeta
+\dotsb+
\sum_{\sigma\in \mathfrak{S}_n}\xi_{\sigma(1)}\zeta \dotsm \xi_{\sigma(n-1)}\xi_{\sigma(n)}
\Bigr]\\
=
(n-1)!
\sum_{\sigma\in \mathfrak{S}_n}
\Bigl[
\xi_{\sigma(1)}\dotsm \xi_{\sigma(n-1)}\xi_{\sigma(n)}\zeta
+\dotsb+
\xi_{\sigma(1)}\zeta \dotsm \xi_{\sigma(n-1)}\xi_{\sigma(n)}
\Bigr],
\end{multline*}
which gives \eqref{l2}.
In the same way we can check \eqref{l23}.

From~\eqref{r14} we now see that
\begin{align*}
\pi_n(\xi_1,\dotsc,\xi_n)
& =
\sum_{\tau\in \mathfrak{S}_{n-2}}
\Bigl[
\xi_{\tau(1)}\dotsm\xi_{\tau(n-2)}\xi_{n-1}\xi_n+
\xi_{\tau(1)}\dotsm\xi_{n-1}\xi_{\tau(n-2)}\xi_n \\
& \qquad {}+  \dotsb +
\xi_{n-1}\xi_{\tau(1)}\dotsm\xi_{\tau(n-2)}\xi_n\Bigr]\\
&\quad {}+\sum_{\tau\in\mathfrak{S}_{n-2}}
\Bigl[
\xi_{\tau(1)}\dotsm\xi_{\tau(n-2)}\xi_n\xi_{n-1}+
\xi_{\tau(1)}\dotsm\xi_{n-1}\xi_n\xi_{\tau(n-2)}\\ 
& \qquad {}+ \dotsb +
\xi_{n-1}\xi_{\tau(1)}\dotsm\xi_n\xi_{\tau(n-2)}\Bigr]\\
&\quad {}+\sum_{\tau\in\mathfrak{S}_{n-2}}
\Bigl[
\xi_{\tau(1)}\dotsm\xi_n\xi_{\tau(n-2)}\xi_{n-1}+
\xi_{\tau(1)}\dotsm\xi_n\xi_{n-1}\xi_{\tau(n-2)} \\ 
& \qquad+
%\quad\quad\quad\quad
\dotsb+
\xi_{n-1}\xi_{\tau(1)}\dotsm\xi_n\xi_{\tau(n-3)}\xi_{\tau(n-2)}\Bigr]\\
& \quad {} + \dotsb + {}\\
& \quad {}+\sum_{\tau\in \mathfrak{S}_{n-2}}
\Bigl[
\xi_n\xi_{\tau(1)}\dotsm\xi_{\tau(n-2)}\xi_{n-1}+
\xi_n\xi_{\tau(1)}\dotsm\xi_{n-1}\xi_{\tau(n-2)}\\
&\qquad {}+ \dotsb+
\xi_n\xi_{n-1}\xi_{\tau(1)}\dotsm\xi_{\tau(n-3)}\xi_{\tau(n-2)}\Bigr].\\
\end{align*}
Therefore, for each $\sigma\in \mathfrak{S}_{n}$, we have
\begin{align*}
&\pi_n(\xi_{\sigma(1)},\dotsc,\xi_{\sigma(n-1)}\xi_{\sigma(n)},\zeta)\\
= &
\sum_{\tau\in \mathfrak{S}_{n-2}}
\Bigl[
\xi_{\sigma(\tau(1))}\dotsm\xi_{\sigma(\tau(n-2))}\xi_{\sigma(n-1)}\xi_{\sigma(n)}\zeta+
\xi_{\sigma(\tau(1))}\dotsm\xi_{\sigma(n-1)}\xi_{\sigma(n)}\xi_{\sigma(\tau(n-2))}\zeta\\ 
&
\qquad {}+\dotsb+
\xi_{\sigma(n-1)}\xi_{\sigma(n)}\xi_{\sigma(\tau(1))}\dotsm\xi_{\sigma(\tau(n-2))}\zeta\Bigr]\\
& {}+\sum_{\tau\in\mathfrak{S}_{n-2}}
\Bigl[
\xi_{\sigma(\tau(1))}\dotsm\xi_{\sigma(\tau(n-2))}\zeta\xi_{\sigma(n-1)}\xi_{\sigma(n)}+
\xi_{\sigma(\tau(1))}\dotsm\xi_{\sigma(n-1)}\xi_{\sigma(n)}\zeta\xi_{\sigma(\tau(n-2))}\\ 
&
\qquad {}+\dotsb+
\xi_{\sigma(n-1)}\xi_{\sigma(n)}\xi_{\sigma(\tau(1))}\dotsm\zeta\xi_{\sigma(\tau(n-2))}\Bigr]\\ 
&
{}+ \sum_{\tau\in\mathfrak{S}_{n-2}}
\Bigl[
\xi_{\sigma(\tau(1))}\dotsm\zeta\xi_{\sigma(\tau(n-2))}\xi_{\sigma(n-1)}\xi_{\sigma(n)}+
\xi_{\sigma(\tau(1))}\dotsm\zeta\xi_{\sigma(n-1)}\xi_{\sigma(n)}\xi_{\sigma(\tau(n-2))}\\ 
&
\qquad {}+\dotsb+
\xi_{\sigma(n-1)}\xi_{\sigma(n)}\xi_{\sigma(\tau(1))}\dotsm\zeta\xi_{\sigma(\tau(n-3))}\xi_{\sigma(\tau(n-2))}\Bigr]\\
& {}+\dots +{}\\
& + \sum_{\tau\in\mathfrak{S}_{n-2}}
\Bigl[
\zeta\xi_{\sigma(\tau(1))}\dotsm\xi_{\sigma(\tau(n-2))}\xi_{\sigma(n-1)}\xi_{\sigma(n)}+
\zeta\xi_{\sigma(\tau(1))}\dotsm\xi_{\sigma(n-1)}\xi_{\sigma(n)}\xi_{\sigma(\tau(n-2))}\\
&
\qquad {}+\dotsb+
\zeta\xi_{\sigma(n-1)}\xi_{\sigma(n)}\xi_{\sigma(\tau(1))}\dotsm\xi_{\sigma(\tau(n-3))}\xi_{\sigma(\tau(n-2))}\Bigr].\\
\end{align*}
We thus get
\begin{align*}
&\sum_{\sigma\in \mathfrak{S}_n}\pi_n (\xi_{\sigma(1)},\dotsc,\xi_{\sigma(n-1)}\xi_{\sigma(n)},\zeta)
\\
=\sum_{\tau\in\mathfrak{S}_{n-2}}
\Bigl[ &
      \sum_{\sigma\in \mathfrak{S}_n}\xi_{\sigma(\tau(1))}\dotsm\xi_{\sigma(\tau(n-2))}\xi_{\sigma(n-1)}\xi_{\sigma(n)}\zeta \\
      & {}+ \sum_{\sigma\in \mathfrak{S}_n}\xi_{\sigma(\tau(1))}\dotsm\xi_{\sigma(n-1)}\xi_{\sigma(n)}\xi_{\sigma(\tau(n-2))}\zeta \\
      & {}+\dotsb+ \sum_{\sigma\in \mathfrak{S}_n}\xi_{\sigma(n-1)}\xi_{\sigma(n)}\xi_{\sigma(\tau(1))}\dotsm\xi_{\sigma(\tau(n-2))}\zeta\Bigr] \\
      {}+ \sum_{\tau\in\mathfrak{S}_{n-2}}
\Bigl[ &
      \sum_{\sigma\in\mathfrak{S}_n}\xi_{\sigma(\tau(1))}\dotsm\xi_{\sigma(\tau(n-2))}\zeta\xi_{\sigma(n-1)}\xi_{\sigma(n)}\\
      & {} + \sum_{\sigma\in\mathfrak{S}_n}\xi_{\sigma(\tau(1))}\dotsm\xi_{\sigma(n-1)}\xi_{\sigma(n)}\zeta\xi_{\sigma(\tau(n-2))} \\
      & {}+\dotsb+ \sum_{\sigma\in\mathfrak{S}_n}\xi_{\sigma(n-1)}\xi_{\sigma(n)}\xi_{\sigma(\tau(1))}\dotsm\zeta\xi_{\sigma(\tau(n-2))}\Bigr] \\
      {}+{} \sum_{\tau\in\mathfrak{S}_{n-2}}
\Bigl[ &
       \sum_{\sigma\in\mathfrak{S}_n}\xi_{\sigma(\tau(1))}\dotsm\zeta\xi_{\sigma(\tau(n-2))}\xi_{\sigma(n-1)}\xi_{\sigma(n)} \\
       & {}+ \sum_{\sigma\in\mathfrak{S}_n}\xi_{\sigma(\tau(1))}\dotsm\zeta\xi_{\sigma(n-1)}\xi_{\sigma(n)}\xi_{\sigma(\tau(n-2))} \\
       & {}+\dotsb+ \sum_{\sigma\in\mathfrak{S}_n}\xi_{\sigma(n-1)}\xi_{\sigma(n)}\xi_{\sigma(\tau(1))}\dotsm\zeta\xi_{\sigma(\tau(n-3))}
			\xi_{\sigma(\tau(n-2))}\Bigr] \\
		& {}+\dotsb +{} \\
%\vdots\quad\quad& \\
{}+\sum_{\tau\in\mathfrak{S}_{n-2}}
\Bigl[&
      \sum_{\sigma\in\mathfrak{S}_n}\zeta\xi_{\sigma(\tau(1))}\dotsm\xi_{\sigma(\tau(n-2))}\xi_{\sigma(n-1)}\xi_{\sigma(n)} \\
      & {}+ \sum_{\sigma\in\mathfrak{S}_n}\zeta\xi_{\sigma(\tau(1))}\dotsm\xi_{\sigma(n-1)}\xi_{\sigma(n)}\xi_{\sigma(\tau(n-2))} \\
      & {}+\dotsb+ \sum_{\sigma\in\mathfrak{S}_n}\zeta\xi_{\sigma(n-1)}\xi_{\sigma(n)}\xi_{\sigma(\tau(1))}\dotsm\xi_{\sigma(\tau(n-3))}\xi_{\sigma(\tau(n-2))}\Bigr]\\
=
\sum_{\tau\in \mathfrak{S}_{n-2}}
\Bigl[
&\sum_{\sigma\in\mathfrak{S}_n}\xi_{\sigma(1)}\dotsm\xi_{\sigma(n-2)}\xi_{\sigma(n-1)}\xi_{\sigma(n)}\zeta+
\sum_{\sigma\in\mathfrak{S}_n}\xi_{\sigma(1)}\dotsm\xi_{\sigma(n-2)}\xi_{\sigma(n-1)}\xi_{\sigma(n)}\zeta\\
& {}+\dotsb+
\sum_{\sigma\in\mathfrak{S}_n}\xi_{\sigma(1)}\xi_{\sigma(2)}\xi_{\sigma(3)}\dotsm\xi_{\sigma(n)}\zeta\Bigr] \\
{}+ \sum_{\tau\in\mathfrak{S}_{n-2}}
\Bigl[
&\sum_{\sigma\in\mathfrak{S}_n}\xi_{\sigma(1)}\dotsm\xi_{\sigma(n-2)}\zeta\xi_{\sigma(n-1)}\xi_{\sigma(n)}+
\sum_{\sigma\in\mathfrak{S}_n}\xi_{\sigma(1)}\dotsm\xi_{\sigma(n-2)}\xi_{\sigma(n-1)}\zeta\xi_{\sigma(n)}\\
& {}+\dotsb +
\sum_{\sigma\in\mathfrak{S}_n}\xi_{\sigma(1)}\xi_{\sigma(2)}\xi_{\sigma(3)}\dotsm\zeta\xi_{\sigma(n)}\Bigr] \\
{}+ \sum_{\tau\in\mathfrak{S}_{n-2}}
\Bigl[
&\sum_{\sigma\in\mathfrak{S}_n}\xi_{\sigma(1)}\dotsm\zeta\xi_{\sigma(n-2)}\xi_{\sigma(n-1)}\xi_{\sigma(n)}+
\sum_{\sigma\in \mathfrak{S}_n}\xi_{\sigma(1)}\dotsm\zeta\xi_{\sigma(n-2)}\xi_{\sigma(n-1)}\xi_{\sigma(n)}\\
& {}+\dotsb + 
\sum_{\sigma\in\mathfrak{S}_n}\xi_{\sigma(1)}\xi_{\sigma(2)}\xi_{\sigma(3)}\dotsm\zeta\xi_{\sigma(n-1)}\xi_{\sigma(n)}\Bigr] \\
{} +\dotsb +{} \\
%\vdots\quad\quad&\\
{}+\sum_{\tau\in\mathfrak{S}_{n-2}}
\Bigl[
&\sum_{\sigma\in\mathfrak{S}_n}\zeta\xi_{\sigma(1)}\dotsm\xi_{\sigma(n-2)}\xi_{\sigma(n-1)}\xi_{\sigma(n)}+
\sum_{\sigma\in\mathfrak{S}_n}\zeta\xi_{\sigma(1)}\dotsm\xi_{\sigma(n-2)}\xi_{\sigma(n-1)}\xi_{\sigma(n)}\\
& {}+\dotsb+
\sum_{\sigma\in\mathfrak{S}_n}\zeta\xi_{\sigma(1)}\xi_{\sigma(2)}\xi_{\sigma(3)}\dotsm\xi_{\sigma(n-1)}\xi_{\sigma(n)}\Bigr]\\
= \sum_{\tau\in\mathfrak{S}_{n-2}}\sum_{\sigma\in\mathfrak{S}_n}
\Bigl[
&(n-1)\xi_{\sigma(1)}\dotsm\xi_{\sigma(n-2)}\xi_{\sigma(n-1)}\xi_{\sigma(n)}\zeta\Bigr]\\
{}+ \sum_{\tau\in\mathfrak{S}_{n-2}}\sum_{\sigma\in\mathfrak{S}_n}
\Bigl[
&\xi_{\sigma(1)}\dotsm\xi_{\sigma(n-2)}\zeta\xi_{\sigma(n-1)}\xi_{\sigma(n)}+
(n-2)
\xi_{\sigma(1)}\dotsm\xi_{\sigma(n-2)}\xi_{\sigma(n-1)}\zeta\xi_{\sigma(n)}\Bigr] \\ 
{}+ \sum_{\tau\in\mathfrak{S}_{n-2}}\sum_{\sigma\in\mathfrak{S}_n}
\Bigl[
& 2\xi_{\sigma(1)}\dotsm\zeta\xi_{\sigma(n-2)}\xi_{\sigma(n-1)}\xi_{\sigma(n)}+
(n-3)
\xi_{\sigma(1)}\dotsm\xi_{\sigma(n-2)}\zeta\xi_{\sigma(n-1)}\xi_{\sigma(n)}\Bigr]\\ {}+ \dotsb +{}  \\
%\vdots\quad\quad&\\
{} +\sum_{\tau\in\mathfrak{S}_{n-2}}\sum_{\sigma\in\mathfrak{S}_n}
\Bigl[
& k\xi_{\sigma(1)}\dotsm\zeta\xi_{\sigma(n-k)}\dotsm\xi_{\sigma(n)}+
(n-k-1)
\xi_{\sigma(1)}\dotsm\xi_{\sigma(n-k)}\zeta\dotsm\xi_{\sigma(n)}\Bigr] \\
{}+\dotsb +{}\\
%\vdots\quad\quad&\\
{}+\sum_{\tau\in\mathfrak{S}_{n-2}}\sum_{\sigma\in\mathfrak{S}_n}
\Bigl[
& (n-1)\zeta\xi_{\sigma(1)}\dotsm\xi_{\sigma(n-2)}\xi_{\sigma(n-1)}\xi_{\sigma(n)}\Bigr]\\
= (n-2)!\sum_{\sigma\in\mathfrak{S}_n}
\Bigl[
& (n-1)\xi_{\sigma(1)}\dotsm\xi_{\sigma(n-2)}\xi_{\sigma(n-1)}\xi_{\sigma(n)}\zeta\Bigr] \\
{}+ (n-2)!\sum_{\sigma\in\mathfrak{S}_n}
\Bigl[
&\xi_{\sigma(1)}\dotsm\xi_{\sigma(n-2)}\zeta\xi_{\sigma(n-1)}\xi_{\sigma(n)}+
(n-2)
\xi_{\sigma(1)}\dotsm\xi_{\sigma(n-2)}\xi_{\sigma(n-1)}\zeta\xi_{\sigma(n)}\Bigr]\\
{}+ (n-2)!\sum_{\sigma\in\mathfrak{S}_n}
\Bigl[
& 2\xi_{\sigma(1)}\dotsm\zeta\xi_{\sigma(n-2)}\xi_{\sigma(n-1)}\xi_{\sigma(n)}+
(n-3)
\xi_{\sigma(1)}\dotsm\xi_{\sigma(n-2)}\zeta\xi_{\sigma(n-1)}\xi_{\sigma(n)}\Bigr]\\
{}+ \dotsb + {}\\
%\vdots\quad\quad&\\
+ (n-2)!\sum_{\sigma\in\mathfrak{S}_n}
\Bigl[
& k\xi_{\sigma(1)}\dotsm\zeta\xi_{\sigma(n-k)}\dotsm\xi_{\sigma(n)}+
(n-k-1)
\xi_{\sigma(1)}\dotsm\xi_{\sigma(n-k)}\zeta\dotsm\xi_{\sigma(n)}\Bigr] \\
{}+\dotsb+{}\\
%\vdots\quad\quad&\\
+(n-2)!\sum_{\sigma\in\mathfrak{S}_n}
\Bigl[
&(n-1)\zeta\xi_{\sigma(1)}\dotsm\xi_{\sigma(n-2)}\xi_{\sigma(n-1)}\xi_{\sigma(n)}\Bigr],
\end{align*}
which gives \eqref{l4}.
\end{proof}

\begin{theorem}\label{main}
Let $X$ and $Y$ be Banach spaces, and
let $P\colon\mathcal{A}(X)\to Y$ be a continuous $n$-homogeneous polynomial.
Suppose that $X^*$ has the bounded approximation property.
Then the following conditions are equivalent:
\begin{enumerate}
\item
the polynomial $P$ is orthogonally additive, i.e.,
$P(S+T)=P(S)+P(T)$ whenever $S$, $T\in\mathcal{A}(X)$ are such that $ST=TS=0$;
\item
the polynomial $P$ is orthogonally additive on $\mathcal{F}(X)$, i.e.,
$P(S+T)=P(S)+P(T)$ whenever $S$, $T\in\mathcal{F}(X)$ are such that $ST=TS=0$;
\item
there exists a unique continuous linear map $\Phi\colon\mathcal{A}(X)\to Y$ such that
$P(T)=\Phi(T^n)$ for each $T\in\mathcal{A}(X)$.
\end{enumerate}
\end{theorem}

\begin{proof}
It is clear that $(1)\Rightarrow(2)$ and that $(3)\Rightarrow(1)$.
We will henceforth prove that $(2)\Rightarrow(3)$.

Since $P$ is orthogonally additive on $\mathcal{F}(X)$, Corollary~\ref{crf} yields
a linear map $\Phi_0\colon\mathcal{F}(X)\to Y$ such that
\begin{equation}\label{crf1}
P(T)=\Phi_0(T^n) \quad (T\in\mathcal{F}(X)).
\end{equation}
It seems to be appropriate to emphasize that we don't know whether or not $\Phi_0$ is continuous, despite the continuity of $P$.

Let $\varphi\colon\mathcal{A}(X)^n\to Y$ be the continuous symmetric $n$-linear map associated with $P$.
By considering the natural embedding of $Y$ into the second dual $Y^{**}$ we regard $Y$ as a closed linear
subspace of $Y^{**}$, and we regard $\varphi$ as a continuous $n$-linear map
from $\mathcal{A}(X)^n$ into $Y^{**}$ henceforth.
Since $X^*$ has the bounded approximation property, it follows that $\mathcal{A}(X)$
has a bounded approximate identity $(E_\lambda)_{\lambda\in\Lambda}$
(\cite[Theorem~2.9.37(iii) and (v)]{D}) and we may assume that
$(E_\lambda)_{\lambda\in\Lambda}$ lies in $\mathcal{F}(X)$.
Let $\varphi'\colon\mathcal{A}(X)^{n-1}\to Y^{**}$ be as defined in the beginning of this section.
The polarization of the identity $\eqref{crf1}$ yields
\begin{equation}\label{e1547}
\varphi(T_1,\dotsc,T_n)=
\frac{1}{n!}\Phi_0\bigl(\pi_n(T_1,\dotsc,T_n)\bigr) \quad ((T_1,\dotsc,T_n)\in\mathcal{F}(X)^n)
\end{equation}
where $\pi_n$ is the noncommutative polynomial introduced in Lemma~\ref{rfc3}.

We claim that
\begin{equation}\label{f1720}
\varphi(T_1,\dotsc,T_n)=
\frac{1}{n!}
\sum_{\sigma\in\mathfrak{S}_n}
\varphi'(T_{\sigma(1)},\dotsc,T_{\sigma(n-1)}T_{\sigma(n)})
\end{equation}
for each $(T_1,\dotsc,T_n)\in\mathcal{A}(X)^n$.
Since $\mathcal{A}(X)$ is the closed linear span of the rank-one operators and both $\varphi$ and $\varphi'$
are continuous, we are reduced to proving~\eqref{f1720} in the case where $T_1,\dotsc,T_n$ are rank-one operators.
Let $T_1,\dotsc,T_n\in\mathcal{F}(X)$ be rank-one operators, and let $\lambda\in\Lambda$.
On account of \eqref{l2}, \eqref{l23}, \eqref{l4}, and \eqref{e1547}, we have
\begin{multline}\label{l15}
\sum_{\sigma\in\mathfrak{S}_n}\varphi(T_{\sigma(1)},\dotsc,T_{\sigma(n)}E_\lambda)\\ 
= \frac{1}{n}\Phi_0\left(\sum_{\sigma\in\mathfrak{S}_n}
\Bigl[
T_{\sigma(1)}\dotsm T_{\sigma(n-1)}T_{\sigma(n)}E_\lambda
+\dotsb+
T_{\sigma(1)}E_\lambda \dotsm T_{\sigma(n-1)}T_{\sigma(n)}
\Bigr]\right) ,
\end{multline}
\begin{multline}\label{l16}
\sum_{\sigma\in\mathfrak{S}_n}\varphi(T_{\sigma(1)},\dotsc,E_\lambda T_{\sigma(n)}) \\ 
= \frac{1}{n}\Phi_0\left(\sum_{\sigma\in\mathfrak{S}_n}
\Bigl[
T_{\sigma(1)}\dotsm T_{\sigma(n-1)}E_\lambda T_{\sigma(n)}
+\cdots+
E_\lambda T_{\sigma(1)} \dotsm T_{\sigma(n-1)}T_{\sigma(n)}
\Bigr]\right),
\end{multline}
\begin{multline}\label{l17}
\sum_{\sigma\in\mathfrak{S}_n}\varphi(T_{\sigma(1)},\dotsc,T_{\sigma(n-1)}T_{\sigma(n)},E_\lambda) \\ 
= \frac{1}{n}\Phi_0\left(\sum_{\sigma\in\mathfrak{S}_n}
\Bigl[
T_{\sigma(1)}\dotsm T_{\sigma(n-1)}T_{\sigma(n)}E_\lambda
+
E_\lambda T_{\sigma(1)} \dotsm T_{\sigma(n-1)}T_{\sigma(n)}
\Bigr]\right)\\
{}+\frac{n-2}{n(n-1)}\Phi_0\left(\sum_{\sigma\in\mathfrak{S}_n}
\Bigl[
T_{\sigma(1)}\dotsm T_{\sigma(n-1)}E_\lambda T_{\sigma(n)}
+\dotsb+
T_{\sigma(1)}E_\lambda \dotsm T_{\sigma(n-1)}T_{\sigma(n)}
\Bigr]\right).
\end{multline}
Adding~\eqref{l15} and~\eqref{l16}, and then subtracting~\eqref{l17} we can assert that
\begin{multline}\label{l18}
\sum_{\sigma\in\mathfrak{S}_n}\varphi(T_{\sigma(1)},\dotsc,T_{\sigma(n)}E_\lambda)+
\sum_{\sigma\in\mathfrak{S}_n}\varphi(T_{\sigma(1)},\dotsc,E_\lambda T_{\sigma(n)})\\ 
{}- \sum_{\sigma\in\mathfrak{S}_n}
\varphi(T_{\sigma(1)},\dotsc,T_{\sigma(n-1)}T_{\sigma(n)},E_\lambda)\\
= \frac{1}{n-1}\Phi_0\left(\sum_{\sigma\in\mathfrak{S}_n}
\Bigl[
T_{\sigma(1)}\dotsm T_{\sigma(n-1)}E_\lambda T_{\sigma(n)}
+\dotsb+
T_{\sigma(1)}E_\lambda \dotsm T_{\sigma(n-1)}T_{\sigma(n)}
\Bigr]\right).
\end{multline}
Let $\mathcal{M}$ be the subalgebra of $\mathcal{F}(X)$ generated by $T_1,\dotsc,T_n$.
Then $\mathcal{M}$ is finite-dimensional, a property which implies that the restriction
of $\Phi_0$ to $\mathcal{M}$ is continuous, of course.
Further, it is straightforward to check that $T_iST_j\in\mathcal{M}$ for all $S\in\mathcal{B}(X)$
and $i,j\in\{1,\dotsc,n\}$. Accordingly, for each $\sigma\in\mathfrak{S}_n$, all the nets
\[
\bigl(T_{\sigma(1)}\dotsb T_{\sigma(n-1)}E_\lambda T_{\sigma(n)}\bigr)_{\lambda\in\Lambda}, \, \dotsc \,
,\bigl(T_{\sigma(1)}E_\lambda\dotsb T_{\sigma(n-1)} T_{\sigma(n)}\bigr)_{\lambda\in\Lambda}
\]
lie in $\mathcal{M}$, and,
since $(E_\lambda)_{\lambda\in\Lambda}$ is a bounded approximate identity for $\mathcal{A}(X)$,
each of them converges to $T_{\sigma(1)}\dotsb T_{\sigma(n-1)}T_{\sigma(n)}$ in the operator norm.
Therefore, taking limits along $\mathcal{U}$ on both sides of the equation \eqref{l18} (and using
the continuity of $\Phi_0$ on $\mathcal{M}$), we see that
\begin{multline}\label{l19}
\sum_{\sigma\in\mathfrak{S}_n}\varphi(T_{\sigma(1)},\dotsc,T_{\sigma(n)})+
\sum_{\sigma\in\mathfrak{S}_n}\varphi(T_{\sigma(1)},\dotsc,T_{\sigma(n)})-
\sum_{\sigma\in\mathfrak{S}_n}
\varphi'(T_{\sigma(1)},\dotsc,T_{\sigma(n-1)}T_{\sigma(n)})\\ =
\frac{1}{n-1}\Phi_0\left(\sum_{\sigma\in\mathfrak{S}_n}
\Bigl[
T_{\sigma(1)}\dotsm T_{\sigma(n-1)} T_{\sigma(n)}
+\dotsb+
T_{\sigma(1)} \dotsm T_{\sigma(n-1)}T_{\sigma(n)}
\Bigr]\right)\\ =
\Phi_0\left(\sum_{\sigma\in\mathfrak{S}_n}
T_{\sigma(1)}\dotsm T_{\sigma(n-1)} T_{\sigma(n)}\right).
\end{multline}
By using the symmetry of $\varphi$ and \eqref{e1547} in \eqref{l19} we obtain
\begin{equation*}
2n!\varphi(T_1,\dotsc,T_n) - \sum_{\sigma\in\mathfrak{S}_n}
\varphi'(T_{\sigma(1)},\dotsc,T_{\sigma(n-1)}T_{\sigma(n)})=
n!\varphi(T_1,\dotsc,T_n),
\end{equation*}
which yields \eqref{f1720}.

Having disposed of identity~\eqref{f1720}, we can now apply
Lemma~\ref{l177} to obtain a continuous linear map $\Phi\colon\mathcal{A}(X)\to Y^{**}$
such that
\begin{equation}\label{aa}
\varphi(T_1,\ldots,T_n)=
\frac{1}{n!}
\Phi\left(\sum_{\sigma\in\mathfrak{S}_n}T_{\sigma(1)}\cdots T_{\sigma(n)}\right)
\end{equation}
for each $(T_1,\ldots,T_n)\in \mathcal{A}(X)$.
Our next objective is to show that the range of $\Phi$ lies actually in $Y$.
Let $T\in\mathcal{F}(X)$. On account of Lemma~\ref{l1525},
there exists $S\in\mathcal{F}(X)$ such that $TS=ST=T$.
From \eqref{aa} we see that
$\Phi(T)=\varphi(T,S,\dotsc,S)\in Y$.
Since $\mathcal{F}(X)$ is dense in $\mathcal{A}(X)$, $\Phi$ is continuous,
and $Y$ is closed in $Y^{**}$, it may be concluded that $\Phi(\mathcal{A}(X))\subset Y$.

Our final task is to prove the uniqueness of the map $\Phi$.
Suppose that $\Psi\colon\mathcal{A}(X)\to Y$ is a continuous linear map
such that $P(T)=\Psi(T^n)$ for each $T\in\mathcal{A}(X)$.
By Corollary~\ref{crf}, $\Psi(T)=\Phi(T)\bigl(=\Phi_0(T)\bigr)$
for each $T\in\mathcal{F}(X)$.
Since $\mathcal{F}(X)$ is dense in $\mathcal{A}(X)$, and both
$\Phi$ and $\Psi$ are continuous, it follows
that $\Psi(T)=\Phi(T)$ for each $T\in\mathcal{A}(X)$.
\end{proof}

\begin{remark}
Let $X$ and $Y$ be as in Theorem~\ref{main}, let $P\colon\mathcal{A}(X)\to Y$ be
a continuous orthogonally additive $n$-homogeneous polynomial, and let $\varphi\colon\mathcal{A}(X)^n\to Y$ be
the continuous symmetric $n$-linear map associated with $P$.
Suppose that  $(E_\lambda)_{\lambda\in\Lambda}$ is any bounded approximate identity for $\mathcal{A}(X)$.
We already know that there exists a continuous linear map  $\Phi\colon\mathcal{A}(X)\to Y$
such that $P(T)=\Phi(T^n)$ for each $T\in\mathcal{A}(X)$. The polarization of this representation gives
\begin{equation}\label{oo}
\varphi(T_1,\dotsc,T_n)=\frac{1}{n!}\Phi\left(\sum_{\sigma\in\mathfrak{S}_n}T_{\sigma(1)}\cdots T_{\sigma(n)}\right)
\end{equation}
for each $(T_1,\ldots,T_n)\in \mathcal{A}(X)$.
If $T\in\mathcal{A}(X)$, then each of the nets
\[
\bigl(TE_\lambda\dotsb E_\lambda\bigr)_{\lambda\in\Lambda},
\bigl(E_\lambda T\dotsb E_\lambda\bigr)_{\lambda\in\Lambda},\dotsc,
\bigl(E_\lambda\dotsb TE_\lambda\bigr)_{\lambda\in\Lambda},
\bigl(E_\lambda\dotsb E_\lambda T\bigr)_{\lambda\in\Lambda},
\]
converges to $T$ in the operator norm, and using \eqref{oo} and the continuity of $\Phi$ we see that the net
$\bigl(\varphi(T,E_\lambda,\dotsc,E_\lambda)\bigr)_{\lambda\in\Lambda}$ converges to $\Phi(T)$ in norm.
Accordingly, the map $\Phi$  is necessarily given by
\begin{equation}\label{ooo}
\Phi(T)=\lim_{\lambda\in\Lambda}\varphi(T,E_\lambda,\dotsc,E_\lambda) \quad (T\in\mathcal{A}(X)),
\end{equation}
where there is no need for taking the limit with respect to any weak topology along any ultrafilter. Note the similarity of \eqref{ooo} with \eqref{m0}.
\end{remark}

\end{document}